\documentclass[10pt]{amsart}
\usepackage{amsmath,amssymb,latexsym,soul,cite}
\usepackage{color,enumitem,graphicx, tikz, mathtools,microtype}
\usepackage[colorlinks=true,urlcolor=blue,citecolor=blue,linkcolor=blue,linktocpage,pdfpagelabels,bookmarksnumbered,bookmarksopen]{hyperref}
\usepackage[english]{babel}
\usepackage[hyperpageref]{backref}
\usepackage[left=3.2cm,right=3.2cm,top=2.9cm,bottom=2.9cm]{geometry}
\usepackage{color,enumitem,graphicx}

\numberwithin{equation}{section}

\newtheorem{thm}{Theorem}[section]
\theoremstyle{plain}
\newtheorem{lem}[thm]{Lemma}
\theoremstyle{plain}

\theoremstyle{plain}

\newtheorem{definition}[thm]{Definition}
\theoremstyle{definition}
\newtheorem{rem}[thm]{Remark}
\newtheorem{ex}[thm]{Example}

\newcommand{\N}{{\mathbb N}}

\newcommand{\R}{{\mathbb R}}

\newcommand{\beq}{\begin{equation}}
\newcommand{\eeq}{\end{equation}}

\renewcommand{\le}{\leqslant}
\renewcommand{\ge}{\geqslant}

\def\XXint#1#2#3{{\setbox0=\hbox{$#1{#2#3}{\int}$ }
\vcenter{\hbox{$#2#3$ }}\kern-.57\wd0}}

\DeclareMathOperator*{\esssup}{ess\, sup}
\DeclareMathOperator*{\essinf}{ess\, inf}

\newcommand{\restr}[2]{\left.#1\right|_{#2}}

\newenvironment{enumroman}{\begin{enumerate}

}{\end{enumerate}}

\title[Differential inclusions with mixed conditions]{On ordinary differential inclusions with mixed boundary conditions}

\author[G.\ Bonanno]{Gabriele Bonanno}
\author[A.\ Iannizzotto]{Antonio Iannizzotto}
\author[M.\ Marras]{Monica Marras}

\address{Department of Civil, Computer, Construction, Environmental Engineering and Applied Mathematics
\newline\indent
University of Messina
\newline\indent
98166 Messina, Italy}
\email{bonanno@unime.it}

\address{Department of Mathematics and Computer Science
\newline\indent
University of Cagliari
\newline\indent
Viale L. Merello 92, 09123 Cagliari, Italy}
\email{antonio.iannizzotto@unica.it, mmarras@unica.it}


\begin{document}

\begin{abstract}
By means of nonsmooth critical point theory, we prove existence of three weak solutions for an ordinary differential inclusion of Sturm-Liouville type involving a general set-valued reaction term depending on a parameter, and coupled with mixed boundary conditions. As an application, we give a multiplicity result for ordinary differential equations involving discontinuous nonlinearities.
\end{abstract}

\maketitle

\begin{center}
Version of \today\
\end{center}

\section{Introduction and main result}\label{sec1}

\noindent
We consider the following second order ordinary differential inclusion (o.d.i.), driven by a Sturm-Liouville type operator, and coupled with mixed boundary conditions:
\beq\label{odi}
\begin{cases}
-(p(x)u')'+q(x)u\in\lambda F(u) &\text{in $(a,b)$} \\
u(a)=u'(b)=0.
\end{cases}
\eeq
Here $a<b$ are real numbers, $p,q\in L^\infty(a,b)$ are s.t.
\[\essinf_{x\in(a,b)}p(x)=p_0>0, \ \essinf_{x\in(a,b)}q(x)\ge 0,\]
while $F:\R\to 2^\R$ is an upper semicontinuous (u.s.c.) set-valued mapping with compact convex values, and $\lambda>0$ is a parameter. O.d.i.'s of the type \eqref{odi} are a very general class of problems, as they extend both hemivariational inequalities and ordinary differential equations (o.d.e.'s), even in implicit form or with discontinuous nonlinearities, as was first noticed by {\sc Filippov} \cite{F1} for a first order problem.
\vskip2pt
\noindent
Most existence results for the solutions of o.d.i.'s are obtained through operator-based methods, such as selection theory, sub- and supersolutions, the theory of monotone operators or fixed point theory. See for instance the papers of {\sc Averna \& Bonanno} \cite{AB}, {\sc Erbe \& Krawcewicz} \cite{EK}, {\sc Frigon \& Granas} \cite{FG}, {\sc Kourogenis} \cite{K} and the monograph of {\sc Aubin \& Frankowska} \cite{AF}. All the mentioned papers, except \cite{AB} and \cite{K}, deal with convex-valued mappings.
\vskip2pt
\noindent
In order to achieve multiplicity results, the best choice seems to be that of applying variational methods. This can be done by exploiting the critical point theory for nonsmooth functional developed by {\sc Clarke} \cite{C1} (see also {\sc Gasi\'{n}ski \& Papageorgiou} \cite{GP}). Many authors, starting from the classical work of {\sc Chang} \cite{C}, have applied nonsmooth analysis to set-valued problems, often arising from either partial or ordinary differential equations with discontinuous nonlinearities, see for instance {\sc Bonanno \& Buccellato} \cite{BB}, {\sc Frigon} \cite{F}, {\sc Iannizzotto} \cite{I,I1}, {\sc Krastanov, Ribarska \& Tsachev} \cite{KRT}, {\sc Papageorgiou \& Papalini} \cite{PP}. Among the mentioned papers, \cite{F}, \cite{I1}, \cite{I}, and \cite{KRT} are concerned with general differential inclusions, while the others are mainly concerned with the case where the set-valued term is the subdifferential of a convenient non smooth potential. For a different viewpoint on set-valued problems seen in a variational framework, see also {\sc \'{C}wiszewski \& Kryszewski} \cite{CK}.
\vskip2pt
\noindent
Here we consider for the first time (to the best of our knowledge) a general o.d.i. with mixed boundary conditions. First, we develop a general variational framework for problem \eqref{odi}, and in doing so we extend the ideas of some previous works. Subsequently, by applying a three critical point for nonsmooth functionals due to {\sc Bonanno \& Marano} \cite{BM}, we prove existence of at least three solutions of problem \eqref{odi} for all $\lambda$ within a precisely determined interval. Our result extends that of {\sc Averna, Giovannelli \& Tornatore} \cite{AGT} to the set-valued case.
\vskip2pt
\noindent
A special case of our main result is the following:

\begin{thm}\label{thm-min}
Let $F:\R\to2^\R$ be u.s.c. with compact convex values, and 
$\alpha>0$, $s\in(1,2)$, $0<c<d$ s.t.
\begin{enumroman}
\item\label{thm-min1} $0\le\min F(t)\le\alpha(1+|t|^{s-1})$ for all $t\in\R$;
\item\label{thm-min2} $\displaystyle\frac{1}{c^2}\int_0^c\min F(t)\,dt<\frac{K}{d^2}\int_0^d \min F(t)\,dt$, with $K=3p_0\big(12\|p\|_\infty+4(b-a)^2\|q\|_\infty\big)^{-1}$.
\end{enumroman}
Moreover, set
\[\Lambda=\left(\frac{p_0d^2}{2K(b-a)^2}\Big(\int_0^d \min F(t)\,dt\Big)^{-1}, \ \frac{p_0 c^2}{2(b-a)^2}\Big(\int_0^c \min F(t)\,dt\Big)^{-1}\right).\]
Then, for all $\lambda\in\Lambda$ problem \eqref{odi} has at least three solutions.
\end{thm}

\noindent
We present two examples of set-valued mappings satisfying all hypotheses of Theorem \ref{thm-min}:

\begin{ex}\label{ex1}
Set $a=0$, $b=1$, $p(x)=q(x)=1$ for all $x\in [0,1]$ (so that $K=3/16$), and define $F:\R\to 2^\R$ by setting
\[F(t)=\begin{cases}
\{0\} & \text{if $t\le 0$} \\
[t^2,\sqrt{t}] & \text{if $0<t<1$} \\
[\sqrt{t},t^2] & \text{if $t\ge 1$.}
\end{cases}\]
Then, $F$ satisfies the hypotheses of Theorem \ref{thm-min} with $\alpha=1$, $s=3/2$, $c\in(0,3/16)$, and $d=1$.
\end{ex}

\noindent
The paper is organized as follows: in Section \ref{sec2} we recall some basic notions of set-valued analysis and nonsmooth critical point theory; in Section \ref{sec3}  we establish a variational framework for problem \eqref{odi} under general assumptions; in Section \ref{sec4} we prove our main results; and in Section \ref{sec5} we present an application to an o.d.e. with a discontinuous nonlinearity.
\vskip4pt
\noindent
{\bf Notation:} Throughout the paper, $C$ will denote a positive constant, whose value may change from case to case. The standard measure used in the paper in Lebesgue, except when otherwise specified. Moreover, in defining intervals like $\Lambda$ above, we shall use the convention $0^{-1}=\infty$.

\section{Some recalls of set-valued and nonsmooth analysis}\label{sec2}

\noindent
We recall some basic notions from set-valued analysis (for details see \cite{AF}). Let $X$, $Y$ be topological spaces, $F:X\to 2^Y$ be a set-valued mapping. $F$ is {\em upper semicontinuous (u.s.c.)} if, for any open set $A\subseteq Y$, the set
\[F^+(A)=\{x\in X:\,F(x)\subseteq A\}\]
is open in $X$. The following lemma is well known, but we prove it for the reader's convenience:

\begin{lem}\label{usc-cc}
If $F:\R\to 2^\R$ is a set-valued mapping with compact convex values, then the following are equivalent:
\begin{enumroman}
\item\label{usc-cc1} $F$ is u.s.c.;
\item\label{usc-cc2} $\min F,\,\max F:\R\to\R$ are l.s.c., u.s.c. respectively as single-valued mappings.
\end{enumroman}
\end{lem}
\begin{proof}
We first prove that \ref{usc-cc1} implies \ref{usc-cc2}. Fix $M\in\R$. The super-level set
\[\{t\in\R:\,\min F(t)>M\}=F^+(M,\infty)\]
is open, hence $\min F$ is l.s.c. In a similar way we prove that $\max F$ is u.s.c.
\vskip2pt
\noindent
Now we prove that \ref{usc-cc2} implies \ref{usc-cc1}. Let $I\subset\R$ be a bounded open interval. Then, the set
\[F^+(I)=\{t\in\R:\,\max F(t)<\sup I\}\cap\{t\in\R:\,\min F(t)>\inf I\}\]
is open. Now, let $A\subseteq\R$ be an open set. We denote by $\mathcal{I}$ the family of bounded open intervals $I\subseteq A$, hence clearly
\[\bigcup_{I\in\mathcal{I}}I=A.\]
For all $t\in F^+(A)$, the set $F(t)\subset A$ is convex and compact, hence there is $I\in\mathcal{I}$ s.t. $F(t)\subset I$. So the set
\[F^+(A)=\bigcup_{I\in\mathcal{I}}F^+(I)\]
is open, and $F$ turns out to be u.s.c.
\end{proof}

\noindent
A single-valued mapping $f:X\to Y$ is a {\em selection} of $F$ if $f(x)\in F(x)$ for all $x\in X$. If $F:\R\to 2^\R$, the Aumann-type integral is a defined by
\beq\label{aumann}
\int_0^t F(\tau)\,d\tau=\Big\{\int_0^t f(\tau)\,d\tau:\,f:\R\to\R \ \text{measurable selection of $F$}\Big\}
\eeq
(note that, if $F$ is u.s.c. with compact convex values, then the integral is well defined, as both $\min F$ and $\max F$ are Baire measurable by Lemma \ref{usc-cc}, and it has convex, compact values by \cite[Theorem 8.6.3]{AF}).
\vskip2pt
\noindent
Now we recall some notions of nonsmooth critical point theory (for details see \cite{GP}). Let $(X,\|\cdot\|)$ be a Banach space, $(X^*,\|\cdot\|_*)$ be its topological dual, and $I:X\to\R$ be a functional. $I$ is said to be {\em locally Lipschitz continuous} if for every $u\in X$ there exist a neighborhood $U$ of $u$ and $L>0$ such that 
\[|I(v)-I(w)|\leq L\|v-w\| \ \mbox{for all $v,w\in U$.}\]
The {\em generalized directional derivative} of $I$ at $u$ along $v\in X$ is
\[I^\circ(u;v)=\limsup_{w\to u\above 0pt t\to 0^+}\frac{I(w+tv)-I(w)}{t}.\]
The {\em generalized subdifferential} of $I$ at $u$ is the set
\[\partial I(u)=\left\{u^*\in X^*:\, \langle u^*,v\rangle\leq I^\circ(u;v) \ \mbox{for all $v\in X$}\right\}.\]

\begin{lem}\label{gd}
If $I,\,J:X\to\R$ are locally Lipschitz continuous, then
\begin{enumroman}
\item\label{gd1} $I^\circ(u;\cdot)$ is positively homogeneous, sub-additive and continuous for all $u\in X$;
\item\label{gd2} $I^\circ(u;-v)=(-I)^\circ(u;v)$ for all $u,v\in X$;
\item\label{gd3} if $I\in C^1(X)$, then $I^\circ(u;v)=I'(u)(v)$ for all $u,v\in X$;
\item\label{gd4} $(I+J)^\circ(u;v)\leq I^\circ(u;v)+J^\circ(u;v)$ for all $u,v\in X$.
\end{enumroman}
\end{lem}

\begin{lem}\label{gg}
If $I,J:X\to\R$ are locally Lipschitz continuous, then
\begin{enumroman}
\item\label{gg1} $\partial I(u)$ is convex, closed and weakly$^*$ compact for all $u\in X$;
\item\label{gg2} $\partial I:X\to 2^{X^*}$ is an upper semicontinuous set-valued mapping with respect to the weak$^*$ topology on $X^*$;
\item\label{gg3} if $I\in C^1(X)$, then $\partial I(u)=\{I'(u)\}$ for all $u\in X$;
\item\label{gg4} $\partial(\lambda I)(u)=\lambda\partial I(u)$ for all $\lambda\in\R$, $u\in X$;
\item\label{gg5} $\partial(I+J)(u)\subseteq\partial I(u)+\partial J(u)$ for all $u\in X$;
\item\label{gg6} for all $u,v\in X$ there exists $u^*\in\partial I(u)$ such that  $u^*(v)=I^\circ(u;v)$;
\item\label{gg7} if $u$ is a local minimizer (or maximizer) of $I$, then $0\in\partial I(u)$.
\end{enumroman}
\end{lem}

\noindent
By Lemma \ref{gg} \ref{gg1}, we may define for all $u\in X$
\begin{equation}\label{m}
m(u)=\min_{u^*\in\partial I(u)}\|u^*\|_*,
\end{equation}
We say that $u\in X$ is a {\em (generalized) critical point} of $I$ if $m(u)=0$ (i.e. $0\in\partial I(u)$). We say that $I$ satisfies the nonsmooth {\em Palais-Smale condition} (for short {\bf PS}) if every sequence $(u_n)$ in $X$, s.t.  $(I(u_n))$ is bounded in $\R$ and $m(u_n)\to 0$, admits a convergent subsequence.
\vskip2pt
\noindent
Nonsmooth critical point theory is by now widely developed, as it includes extensions of most well-known results in classical critical point theory for $C^1$ functionals (such as the mountain pass theorem, deformation lemmas, and Morse theory). We will make use of the following three critical points theorem due to {\sc Bonanno \& Marano} \cite{BM} (here rephrased for the reader's convenience):

\begin{thm}\label{3cp}
Let $(X,\|\cdot\|)$ be a reflexive Banach space, $\Phi, \Psi:X\to\R$ be locally Lipschitz continuous functionals, set for all $r\neq 0$
\[\varphi(r)=\sup_{\Phi(u)\le r}\frac{\Psi(u)}{r}\]
and $I_\lambda=\Phi-\lambda\Psi$ for all $\lambda>0$. Assume that
\begin{enumroman}
\item\label{3cp1} $\Phi$ is sequentially weakly l.s.c. and coercive;
\item\label{3cp2} $\Psi$ is sequentially weakly u.s.c.;
\item\label{3cp3} $I_\lambda$ satisfies {\bf PS} for all $\lambda>0$;
\item\label{3cp4} there exist $r>\inf_X\Phi$, $\bar u\in X$ s.t. $\Phi(\bar u)>r$, $\varphi(r)<\Psi(\bar u)/\Phi(\bar u)$.
\end{enumroman}
Then, for all $\lambda\in\big(\Phi(\bar u)/\Psi(\bar u),1/\varphi(r)\big)$ the functional $I_\lambda$ admits at least three critical points in $X$.
\end{thm}


\section{Variational framework}\label{sec3}

\noindent
This section is mainly devoted to establishing a variational framework for problem \eqref{odi} under very general assumptions. We consider mixed boundary conditions, but our framework can easily be adapted to Dirichlet, Neumann, or periodic conditions as well as to the case of non-autonomous reaction terms. We generalize the approach of \cite{I} and of other works on the subject, which will be recalled below (see Remark \ref{prev}).
\vskip2pt
\noindent
Our assumptions on the set-valued mapping $F$ are the following:
\begin{itemize}[leftmargin=0.7cm]
\item[${\bf H}_0$] $F:\R\to 2^\R$ is u.s.c. with compact convex values and admits a Baire measurable selection $f:\R\to\R$ s.t. for all $t\in\R$
\[|f(t)|\le\alpha(1+|t|^{s-1}) \ (\alpha>0,\,s>1).\]
\end{itemize}
We define a convenient function space (for details see {\sc Brezis} \cite{B}):
\[X=\{u\in H^1(a,b):\,u(a)=0\}, \ \|u\|=\Big(\int_a^b (p(x)(u')^2+q(x)u^2)\,dx\Big)^\frac{1}{2}.\]
Due to the positivity of $p$ and non-negativity of $q$, it is easily seen that $\|\cdot\|$ is a norm on $X$ and it is equivalent to the $H^1(a,b)$-norm restricted to $X$. Moreover, $(X,\|\cdot\|)$ is a Hilbert space with inner product
\[\langle u,v\rangle=\int_a^b(p(x)u'v'+q(x)uv)\,dx.\]
We also note that the embedding $X\hookrightarrow C^0([a,b])$ is compact and for all $u\in X$
\beq\label{emb}
\|u\|_\infty\le\Big(\frac{b-a}{p_0}\Big)^\frac{1}{2}\|u\|
\eeq
(by $\|\cdot\|_\nu$ we denote the norm of $L^\nu(a,b)$, for any $\nu\in[1,\infty]$). We seek solutions in the space $X$:

\begin{definition}\label{sol}
We say that $u\in X$ is a (weak) solution of \eqref{odi}, if there exists $w\in L^\nu(a,b)$ ($\nu>1$) s.t.
\begin{enumroman}
\item\label{sol1} $\displaystyle\langle u,v\rangle=\lambda\int_a^b wv\,dx$ for all $v\in X$;
\item\label{sol2} $w(x)\in F(u(x))$ for a.e. $x\in(a,b)$.
\end{enumroman}
\end{definition}

\begin{rem}\label{class}
Definition \ref{sol} is quite natural, as it agrees with an intuitive notion of 'classical solution' provided the involved functions are sufficiently smooth. Indeed, let $p\in C^1([a,b])$, $q\in C^0([a,b])$, and $u\in X\cap C^2([a,b])$ be a (weak) solution of \eqref{odi}. By \ref{sol1} and integration by parts, we have for all $v\in X$
\[\int_a^b\big(-(p(x)u')'v+(p(x)u'v)'+q(x)uv\big)\,dx=\lambda\int_a^b wv\,dx.\]
Taking an arbitrary $v\in H^1_0(a,b)$ we see that
\[-(p(x)u')'+q(x)u=\lambda w \ \text{in $(a,b)$,}\]
which in turn implies for any $v\in X$
\[0=p(a)u'(a)v(a)=p(b)u'(b)v(b),\]
hence $u'(b)=0$. Thus, by \ref{sol2}, $u$ solves \eqref{odi} in a pointwise sense.
\end{rem}

\noindent
For all $t\in\R$ we set
\beq\label{jf}
J_f(t)=\int_0^t f(\tau)\,d\tau,
\eeq
and we define two functionals by setting for all $u\in X$
\[\Phi(u)=\frac{\|u\|^2}{2}, \ \Psi(u)=\int_a^b J_f(u)\,dx.\]
The following lemma displays some easy properties of $\Phi$:

\begin{lem}\label{phi}
The functional $\Phi\in C^1(X)$ is coercive, weakly l.s.c., and for all $u,v\in X$
\[\Phi'(u)(v)=\langle u,v\rangle.\]
\end{lem}

\noindent
The next lemma is the most delicate part of our method. It is inspired by the results of \cite{C}:

\begin{lem}\label{psi}
If hypotheses ${\bf H}_0$ hold, then the functional $\Psi:X\to\R$ is locally Lipschitz continuous, sequentially weakly continuous, and for all $u\in X$, $w^*\in\partial\Psi(u)$ there exists $w\in L^{s'}(a,b)$ ($s'=s/(s-1)$) s.t.
\begin{enumroman}
\item\label{psi1} $\displaystyle w^*(v)=\int_a^b wv\,dx$ for all $v\in X$;
\item\label{psi2} $w(x)\in F(u(x))$ for a.e. $x\in(a,b)$.
\end{enumroman}
\end{lem}
\begin{proof}
First we note that $J_f:\R\to\R$ is locally Lipschitz continuous, and by \cite[Example 1]{C} we have for all $t\in\R$
\beq\label{dej}
\partial J_f(t)\subseteq\big[\underline f(t),\,\overline f(t)\big],
\eeq
where we have set
\beq\label{appr}
\underline f(t)=\lim_{\delta\to 0^+}\essinf_{|h|<\delta}f(t+h), \ \overline f(t)=\lim_{\delta\to 0^+}\esssup_{|h|<\delta}f(t+h).
\eeq
We also have for all $t\in\R$
\beq\label{mm}
\min F(t)\le\underline f(t)\le\overline f(t)\le\max F(t).
\eeq
Indeed, for all $\delta>0$ we have
\[\essinf_{|h|<\delta}f(t+h)\ge\inf_{0<|h|<\delta}\min F(t+h).\]
Letting $\delta\to 0^+$, and recalling that $\min F$ is l.s.c., we have
\[\underline f(t)\ge\min F(t).\]
The other inequality of \eqref{mm} is achieved in a similar way.
\vskip2pt
\noindent
Now we set for all $u\in L^s(a,b)$
\[\tilde\Psi(u)=\int_a^b J_f(u)\,dx.\]
By ${\bf H}_0$, $\tilde\Psi:L^s(a,b)\to\R$ is well defined. Besides, $\tilde\Psi$ is Lipschitz continuous on any bounded subset of $L^s(a,b)$. Indeed, for all $M>0$, $u,v\in L^s(a,b)$ with $\|u\|_s,\|v\|_s\le M$, by ${\bf H}_0$ and H\"older inequality we have
\begin{align*}
|\tilde\Psi(u)-\tilde\Psi(v)| &\le \int_a^b\Big|\int_u^v f(\tau)\,d\tau\Big|\,dx \\
&\le \alpha\int_a^b\big(1+|u|^{s-1}+|v|^{s-1}\big)|u-v|\,dx \\
&\le C(1+M^{s-1})\|u-v\|_s.
\end{align*}
Now fix $u\in L^s(a,b)$, $w^*\in\partial\tilde\Psi(u)$. By \cite[Theorem 4.11]{B} we can find $w\in L^{s'}(a,b)$ s.t. for all $v\in L^s(a,b)$
\[w^*(v)=\int_a^b wv\,dx.\]
By \cite[Theorem 2.1]{C} we have $w(x)\in\partial J_f(u(x))$ a.e. in $(a,b)$, which by \eqref{dej} and \eqref{mm} implies
\[w(x)\in\big[\min F(u(x)),\max F(u(x))\big] \ \text{a.e. in $(a,b)$.}\]
Recalling that $F(u(x))$ is a convex set, we finally get $w(x)\in F(u(x))$ a.e. in $(a,b)$.
\vskip2pt
\noindent
Now we turn back to $\Psi$. Thanks to the embedding $X\hookrightarrow L^s(a,b)$, we may identify $X$ with a linear subspace of $L^s(a,b)$. By using \eqref{emb}, we easily see that $\Psi=\restr{\tilde\Psi}{X}$ is well defined and locally Lipschitz continuous. Moreover, let $u\in X$ and $w^*\in\partial\Psi(u)$. By definition of $\partial\Psi$, we have for all $v\in X$
\[w^*(v)\le\Psi^\circ(u;v)\le\tilde\Psi^\circ(u;v),\]
and by Lemma \ref{gd} \ref{gd1} $\tilde\Psi^\circ(u;\cdot):L^s(a,b)\to\R$ is a semi-norm. By the Hahn-Banach theorem, there exists $\tilde w^*\in (L^s(a,b))^*$ s.t. $\restr{\tilde w^*}{X}=w^*$ and for all $v\in L^s(a,b)$
\[\tilde w^*(v)\le\tilde\Psi^\circ(u;v).\]
So, we have $\tilde w^*\in\partial\tilde\Psi(u)$, and reasoning as above we can find $w\in L^{s'}(a,b)$ satisfying \ref{psi1} and \ref{psi2} (this argument improves \cite[Theorem 2.2]{C}, as it requires no density assumption).
\vskip2pt
\noindent
Finally, we prove that $\Psi$ is sequentially weakly continuous in $X$. Indeed, if $u_n\rightharpoonup u$ in $X$, then $(u_n)$ is bounded in $X$ and, passing if necessary to a subsequence, we have $u_n\to u$ in $C^0([a,b])$, hence $\Psi(u_n)\to\Psi(u)$. Then, we easily retrieve $\Psi(u_n)\to\Psi(u)$ for the original sequence.
\end{proof}

\noindent
For all $\lambda>0$ we set $I_\lambda=\Phi-\lambda\Psi$, thus defining an energy functional for problem \eqref{odi}. The following lemma displays the main properties of $I_\lambda$.

\begin{lem}\label{il}
If hypotheses ${\bf H}_0$ hold, then for all $\lambda>0$ the functional $I_\lambda:X\to\R$ is locally Lipschitz continuous and if $u\in X$ is a critical point of $I_\lambda$, then $u$ is a solution of \eqref{odi}. Moreover, every bounded {\bf PS}-sequence for $I_\lambda$ has a convergent subsequence in $X$. 
\end{lem}
\begin{proof}
By Lemmas \ref{gg} \ref{gg5}, \ref{phi}, and \ref{psi} we easily see that $I_\lambda$ is well defined and locally Lipschitz continuous in $X$ and for all $u\in X$, $u^*\in\partial I_\lambda(u)$ there exists $w^*\in\partial\Psi(u)$ s.t. for all $v\in X$
\beq\label{rep}
u^*(v)=\langle u,v\rangle-\lambda w^*(v).
\eeq
In particular, let $u\in X$ be s.t. $0\in\partial I_\lambda(u)$. By \eqref{rep} and Lemma \ref{psi}, there exists $w\in L^{s'}(a,b)$ satisfying \ref{sol1} and \ref{sol2} of Definition \ref{sol}, hence $u$ is a solution of \eqref{odi}.
\vskip2pt
\noindent
Finally, let $(u_n)$ be a bounded sequence in $X$, s.t.  $(I_\lambda(u_n))_n$ is bounded in $\R$ and $m_\lambda(u_n)\to 0$ ($m_\lambda$ defined as in \eqref{m}). By reflexivity of $X$ and the compact embedding $X\hookrightarrow C^0([a,b])$, by passing to a subsequence we have $u_n\rightharpoonup u$ in $X$ and $u_n\to u$ in $C^0([a,b])$. By Lemma \ref{gg} \ref{gg3}--\ref{gg6}, for all $n\in\N$ there exists $w^*_n\in\partial\Psi(u_n)$ s.t.
\[m_\lambda(u_n)=\|u_n-\lambda w^*_n\|_*\]
(where we have used the Riesz theorem to identify $u_n$ with an element of $X^*$). Furthermore, we can find $w_n\in L^{s'}(a,b)$ satisfying \ref{psi1} and \ref{psi2} of Lemma \ref{psi}. Now we exploit the collected information to get the following estimates for all $n\in\N$:
\begin{align*}
\|u_n-u\|^2 &= \langle u_n-u,u_n-u\rangle \\
&= \langle u_n, u_n-u\rangle +o(1) \\
&\le \|u_n-\lambda w^*_n\|_*\|u_n-u\|+\lambda w^*_n(u_n-u)+o(1) \\
&\le m_\lambda(u_n)\|u_n-u\|+\lambda\int_a^b |w_n(u_n-u)|\,dx+o(1) \\
&\le \lambda\alpha\int_a^b(1+|u_n|^{s-1})|u_n-u|\,dx+o(1) \\
&\le \lambda C\|u_n-u\|_\infty+o(1),
\end{align*}
and the latter tends to $0$ as $n\to\infty$. So, $u_n\to u$ in $X$.
\end{proof}

\begin{rem}\label{prev}
Possible choices of the selection $f$ in ${\bf H}_0$ are the following:
\[f_1(t)=\min F(t), \ f_2(t)=\max F(t), \ f_3(t)=\frac{1}{2}\big(\min F(t)+\max F(t)\big).\]
In general, since $\min F$ and $\max F$ are Baire measurable due to Lemma \ref{usc-cc}, any convex, continuous combination of the two yields a further Baire measurable selection of $F$. Also, the following case has been considered in several papers (see \cite{F}, \cite{I}, \cite{KRT}):
\[f_4(t)=\begin{cases}
\max F(t) & \text{if $t<0$} \\
\min F(t) & \text{if $t\ge 0$.}
\end{cases}\]
This is mainly due to historical reasons: indeed in \cite{F} (possibly the earliest paper dealing with general differential inclusions in a variational perspective) the Author applied metric critical point theory to a functional of the type
\[u\mapsto\frac{\|u\|^2}{2}-\int_a^b\min\int_0^u F(\tau)\,d\tau\,dx\]
(where the set-valued integral defined in \eqref{aumann} is involved), and $f_4$ exactly produces the desired identity
\[J_{f_4}(t)=\min\int_0^t F(\tau)\,d\tau\]
for all $t\in\R$. Nevertheless, it seems that there is no intrinsic reason to prefer one of the above choices if $f$, so one basically may choose the most convenient case by case.
\end{rem}

\begin{rem}\label{locsum}
A natural question, in connection with this method, is whether we can replace in ${\bf H}_0$ the right-hand side $\alpha(1+|t|^{s-1})$ by any non-negative, continuous, increasing and convex $\eta(t)$ (for instance $e^t$). This may ensure local Lipschitz continuity of the functional $\Psi$ (by means of the embedding $X\hookrightarrow C^0([a,b])$), but might not allow such a natural representation of the elements of $\partial\Psi(u)$ at any $u\in X$, as in Lemma \ref{psi}.
\end{rem}


\section{Three solutions for an o.d.i.}\label{sec4}

\noindent
In this section we prove our main result, namely the existence of three solutions for problem \eqref{odi}, for all $\lambda$ lying in an explicitly determined interval. Our assumptions on $F$ are the following:
\begin{itemize}[leftmargin=0.7cm]

\item[${\bf H}_1$] $F:\R\to 2^\R$ is u.s.c. with compact convex values, admits a Baire measurable selection $f:\R\to\R$, and there exist $\alpha>0$, $s\in(1,2)$, and $0<c<d$ s.t.
\begin{enumroman}
\item\label{h1} $|f(t)|\le\alpha(1+|t|^{s-1})$ for all $t\in\R$;
\item\label{h2} $J_f(t)\ge 0$ for all $t\in[0,d]$ ($J_f$ defined as in \eqref{jf});
\item\label{h3} $\displaystyle\frac{1}{c^2}\max_{|t|\le c}J_f(t)<\frac{K}{d^2} J_f(d)$ (with $K=3p_0\big(12\|p\|_\infty+4(b-a)^2\|q\|_\infty\big)^{-1}$).
\end{enumroman}
\end{itemize}
Due to ${\bf H}_1$ \ref{h3}, the interval
\beq\label{intlambda}
\Lambda=\left(\frac{p_0d^2}{2K(b-a)^2}J_f(d)^{-1}, \ \frac{p_0 c^2}{2(b-a)^2}\Big(\max_{|t|\le c}J_f(t)\Big)^{-1}\right)
\eeq
is nondegenerate. Our result reads as follows:

\begin{thm}\label{main}
Let hypotheses ${\bf H}_1$ hold. Then, for all $\lambda\in\Lambda$ problem \eqref{odi} has at least three solutions.
\end{thm}
\begin{proof}
First we note that ${\bf H}_1$ clearly implies ${\bf H}_0$, hence we can use all results from Section \ref{sec3}. We define $X$ and functionals $\Phi$, $\Psi$, and $I_\lambda$ as above, and by Lemmas \ref{phi} and \ref{psi} we see that hypotheses \ref{3cp1} and \ref{3cp2} of Theorem \ref{3cp} are satisfied. 
\vskip2pt
\noindent
We see now that $I_\lambda$ is coercive for all $\lambda>0$. Indeed, by ${\bf H}_1$ \ref{h1} ($s<2$) and the continuous embedding $X\hookrightarrow C^0([a,b])$, for all $u\in X$ we have
\begin{align*}
I_\lambda(u) &\ge \frac{\|u\|^2}{2}-\lambda C\int_a^b(|u|+|u|^s)\,dx \\
&\ge \frac{\|u\|^2}{2}-\lambda C\big(\|u\|+\|u\|^s),
\end{align*}
and the latter tends to $\infty$ as $\|u\|\to\infty$.
\vskip2pt
\noindent
As a consequence, $I_\lambda$ satisfies {\bf PS} for all $\lambda>0$. Indeed, let $(u_n)$ be a {\bf PS}-sequence for $I_\lambda$. By coercivity, $(u_n)$ is bounded in $X$, hence by Lemma \ref{il} is has a convergent subsequence. Thus, also hypothesis \ref{3cp3} of Theorem \ref{3cp} holds.
\vskip2pt
\noindent
Furthermore, we note that $\inf_X\Phi=0$. We set
\[r=\frac{c^2p_0}{2(b-a)},\]
and we define $\bar u\in X$ by setting
\[\bar u(x)=\begin{cases}
\displaystyle\frac{2d}{b-a}(x-a) & \text{if $\displaystyle x\in\Big[a,\frac{a+b}{2}\Big[$} \\
d & \text{if $\displaystyle x\in\Big[\frac{a+b}{2},b\Big]$.}
\end{cases}\]
We evaluate our functionals at $\bar u$. First, we note that
\begin{align}\label{phig}
\Phi(\bar u) &= \frac{2d^2}{(b-a)^2}\int_a^\frac{a+b}{2} p(x)\,dx+\frac{2d^2}{(b-a)^2}\int_a^\frac{a+b}{2} q(x)(x-a)^2\,dx+\frac{d^2}{2}\int_\frac{a+b}{2}^b q(x)\,dx \\
&\ge \frac{d^2 p_0}{b-a} \nonumber \\
&> r. \nonumber
\end{align}
Besides, we have
\begin{align}\label{phil}
\Phi(\bar u) &\le \frac{d^2\|p\|_\infty}{b-a}+\frac{d^2(b-a)\|q\|_\infty}{12}+\frac{d^2(b-a)\|q\|_\infty}{4} \\
&= \frac{d^2p_0}{4K(b-a)}, \nonumber
\end{align}
where $K$ is defined as in ${\bf H}_1$ \ref{h3}. Now fix $u\in X$ with $\|u\|\le r$. By \eqref{emb} we have $\|u\|_\infty\le c$, hence
\[\Psi(u)\le (b-a)\max_{|t|\le c}J_f(t),\]
which in turn implies
\beq\label{phir}
\varphi(r)\le\frac{2(b-a)^2}{cp_0}\max_{|t|\le c}J_f(t).
\eeq
By ${\bf H}_1$ \ref{h2} we have
\[\Psi(\bar u)\ge\frac{b-a}{2}J_f(d),\]
which, together with ${\bf H}_1$ \ref{h3} and \eqref{phil}, yields
\begin{align}\label{psiphi}
\frac{\Psi(\bar u)}{\Phi(\bar u)} &\ge \frac{4K(b-a)}{d^2p_0}\int_{\frac{a+b}{2}}^b J_f(d)\,d\tau \\
&= \frac{2K(b-a)^2}{d^2p_0}J_f(d) \nonumber \\
&> \frac{2(b-a)^2}{cp_0}\max_{|t|\le c}J_f(t) \nonumber \\
&\ge \varphi(r). \nonumber
\end{align}
Now, by \eqref{phig} and \eqref{psiphi} we see that assumption \ref{3cp4} of Theorem \ref{3cp} also holds. By \eqref{phir} and \eqref{psiphi} we have
\[\Lambda\subseteq\Big(\frac{\Phi(\bar u)}{\Psi(\bar u)},\frac{1}{\varphi(r)}\Big),\]
where $\Lambda$ is defined by \eqref{intlambda}. Thus, for all $\lambda\in\Lambda$ the functional $I_\lambda$ has at least three critical points in $X$. By Lemma \ref{il}, each of such critical points turns out to be a solution of \eqref{odi}, which concludes the proof.
\end{proof}

\noindent
As a special case, we prove the result stated in the Introduction:
\vskip6pt
\noindent
{\bf Proof of Theorem \ref{thm-min}.} We set for all $t\in\R$
\[f(t)=\min F(t).\]
By Lemma \ref{usc-cc}, $f:\R\to\R$ is a l.s.c. (in particular, Baire measurable) selection of $F$, and assumptions \ref{thm-min1}, \ref{thm-min2} imply that ${\bf H}_1$ hold. Then, the conclusion follows from Theorem \ref{main}. \qed

\section{Application: an o.d.e. with a discontinuous nonlinearity}\label{sec5}

\noindent
As a consequence, our results yields the existence of three nontrivial weak solutions for an o.d.e. with mixed boundary conditions of the following type:
\beq\label{ode}
\begin{cases}
-u''=\lambda g(u) &\text{in $(a,b)$} \\
u(a)=u'(b)=0,
\end{cases}
\eeq
Our assumptions on the nonlinearity $g$ are the following:
\begin{itemize}[leftmargin=0.7cm]

\item[${\bf H}_2$] The mapping $g:\R\to\R$ is almost everywhere continuous and there exist $\alpha>0$, $s\in(1,2)$, and $0<c<d$ s.t.
\begin{enumroman}
\item\label{g1} $g(t)\le\alpha(1+|t|^{s-1})$ for all $t\in\R$;
\item\label{g2} $\displaystyle\essinf_{t\in\R}g(t)>0$;
\item\label{g3} $\displaystyle \frac{J_g(c)}{c^2}<\frac{J_g(d)}{4d^2}$ ($J_g$ defined as in \eqref{jf}).
\end{enumroman}
\end{itemize}
Due to ${\bf H}_2$ \ref{g3}, the interval
\[\Lambda=\left(\frac{2d^2}{(b-a)^2J_g(d)}, \ \frac{c^2}{2(b-a)^2J_g(c)}\right)\]
is nondegenerate. Our result is on the same thread as those of \cite{BB}, with mixed boundary conditions as the main difference (such method was first used by {\sc Marano \& Motreanu} \cite{MM} and it makes a substantial use of a classical result of {\sc De Giorgi, Buttazzo \& Dal Maso} \cite{DBD}):

\begin{thm}\label{app}
Let hypotheses ${\bf H}_2$ hold. Then, for all $\lambda\in\Lambda$ problem \eqref{ode} has at least three nonzero solutions.
\end{thm}
\begin{proof}
First we note that, due to ${\bf H}_2$ \ref{g3}, problem \eqref{ode} does not admit the zero solution.
\vskip2pt
\noindent
We denote
\[D=\{t\in\R:\,g \ \text{is not continuous at $t$}\},\]
so ${\bf H}_2$ implies that $D$ has zero measure. We define $\underline g,\overline g:\R\to\R$ as in \eqref{appr}, hence $\underline g$ is l.s.c. and $\overline g$ is u.s.c. Clearly, we have at any $t\in\R\setminus D$
\beq\label{cont}
\underline g(t)=g(t)=\overline g(t).
\eeq
We set for all $t\in\R$
\[F(t)=\big[\underline g(t),\,\overline g(t)\big],\]
so $F:\R\to 2^\R$ is u.s.c. (Lemma \ref{usc-cc}) with compact convex values. The mapping $f:\R\to\R$ defined by setting for all $t\in\R$
\[f(t)=\frac{\underline g(t)+\overline g(t)}{2}\]
is Baire measurable and satisfies ${\bf H}_1$ \ref{h1} - \ref{h3} with $K=1/4$ (recall that $p=1$ and $q=0$). Also, the interval $\Lambda$ coincides with that defined in \eqref{intlambda}. Thus, by Theorem \ref{main}, problem \eqref{odi} (with $p(x)=1$ and $q(x)=0$ for all $x\in(a,b)$, hence $K=1/4$) has at least three solutions for any $\lambda\in\Lambda$.
\vskip2pt
\noindent
To conclude, we prove that any solution $u\in X$ of \eqref{odi} is in fact a (weak) solution of \eqref{ode}, i.e., that for all $v\in X$
\beq\label{ws}
\int_a^b u'v'\,dx=\lambda\int_a^b g(u)v\,dx.
\eeq
Indeed, we know from Definition \ref{sol} that there is $w\in L^{s'}(a,b)$ s.t. for all $v\in X$
\[\int_a^b u'v'\,dx=\lambda\int_a^b wv\,dx,\]
as well as
\beq\label{wgg}
w(x)\in\big[\underline g(u(x)),\,\overline g(u(x))\big] \ \text{a.e. in $(a,b)$.}
\eeq
We recall that $u\in C^0([a,b])$, and we claim that the set $u^{-1}(D)$ has zero Lebesgue measure. Otherwise, by \cite[Lemma 1]{DBD} we would have $u'(x)=0$ a.e. in $u^{-1}(D)$. So, choosing a non-negative $v\in X$ vanishing outside $u^{-1}(D)$, from \eqref{wgg} we would get
\[\int_{u^{-1}(D)}\underline g(u)v\,dx\le 0,\]
hence $\underline g(u)\le 0$ a.e. in $u^{-1}(D)$, against ${\bf H}_2$ \ref{g2}.
\vskip2pt
\noindent
Thus, by \eqref{cont} and \eqref{wgg} we obtain \eqref{ws}, which concludes the proof.
\end{proof}

\noindent
We conclude by presenting an example of a discontinuous mapping satisfying ${\bf H}_2$:

\begin{ex}\label{ex2}
Set $a=0$, $b=1$, and define $g:\R\to\R$ by setting
\[g(t)=\begin{cases}
e^t & \text{if $t<10$} \\
\{\ln(t)\} & \text{if $t\ge 10$.}
\end{cases}\]
Then, $g$ satisfies ${\bf H}_2$ with $\alpha>0$ big enough, any $s\in(1,2)$, $c=1$, and $d=10$.
\end{ex}

\vskip6pt
\noindent
{\bf Aknowledgement.} The authors are members of the Gruppo Nazionale per l'Analisi Matematica, la Probabilit\`a e le loro Applicazioni (GNAMPA) of the Istituto Nazionale di Alta Matematica (INdAM). This work was partially performed while the second and third authors were visiting the University of Messina, which they both recall with gratitude.

\end{document}